\documentclass[a4paper,12pt,reqno]{amsart}
\usepackage{amscd,amssymb,graphicx,color,epigraph,amsmath}
\usepackage{tree-dvips}
\usepackage{qtree}
\usepackage[shortalphabetic,initials]{amsrefs}
\usepackage{hyperref}

\usepackage[dvipsnames]{xcolor}
\usepackage[T1]{fontenc}
\usepackage{hyperref}

\usepackage{pdflscape}

\usepackage{array}
\usepackage{tabularx}
\usepackage{cellspace}
\newcolumntype{C}[1]{>{\centering\arraybackslash}S{p{#1}}}
\hypersetup{
colorlinks = True,
linkcolor=Bittersweet,
citecolor=red,
}

\usepackage{enumitem}
\setlist[itemize]{leftmargin=18pt}
\setlist[enumerate]{leftmargin=18pt}
\usepackage{comment}

\usepackage{eufrak}
\usepackage{mathrsfs}
\usepackage{caption}
\usepackage{xparse}
\usepackage{amscd}
\usepackage[all,cmtip]{xy}
\usepackage[utf8]{inputenc}
\usepackage[T1]{fontenc}
\usepackage{graphicx}
\usepackage{xcolor}
\usepackage{array}
\usepackage{tikz}
\usetikzlibrary{decorations.pathreplacing, calligraphy}
\usepackage{multirow}
\usepackage{makecell}
\usepackage{float} 

\definecolor{customred}{rgb}{0.82,0.01,0.11}

\usepackage[margin=2.5cm]{geometry}

\usepackage{placeins}
\usepackage{afterpage}
\usepackage{amscd}
\usepackage{float}
\usepackage{graphics}
\usepackage{tikz}
\usetikzlibrary{graphs}
\usetikzlibrary{decorations.pathreplacing,angles,quotes}
\usetikzlibrary{shapes.geometric}

\usepackage{tikz-cd}
\usepackage{comment}
\usepackage{xspace}
\usepackage{mathtools}
\usepackage{pifont} %
\usepackage{amssymb}
\usepackage{amsthm}
\usepackage{graphicx}
\usepackage{subfig}
\usepackage{enumerate}
\usepackage{hyperref}

\usetikzlibrary{patterns,decorations.pathreplacing}
\usepackage{marginnote}

\usepackage{cancel}

\theoremstyle{plain}

\newtheorem{theorem}{Theorem}[section]

\newtheorem{lemma}[theorem]{Lemma}
\newtheorem{corollary}[theorem]{Corollary}

\theoremstyle{definition}

\newcommand{\appsection}[1]{\let\oldthesection\thesection
\renewcommand{\thesection}{Appendix \oldthesection}
\section{#1}\let\thesection\oldthesection}

\newtheorem{definition}[theorem]{Definition}

\theoremstyle{remark}

\newtheorem{remark}[theorem]{Remark}

\def\D{{\mathbb{D}}}

\def\Z{{\mathbb{Z}}}
\def\F{{\mathbb{F}}}
\def\Q{{\mathbb{Q}}}
\def\C{{\mathbb{C}}}
\def\P{{\mathbb{P}}}

\def\O{{\mathcal{O}}}

\def\Y{{\mathcal{Y}}}

\def\W{{\mathcal{W}}}

\usepackage{xcolor}

\newcommand{\QHD}{$\mathbb{Q}$\text{HD}\xspace}

\title{Degenerations of the complex projective plane with only rational singularities}

\author{Marcos Canedo}
\address{Facultad de Matem\'aticas,
Pontificia Universidad Cat\'olica de Chile, Santiago, Chile.}
\email{mgcanedo@uc.cl}

\author{Giancarlo Urz\'ua}
\address{Facultad de Matem\'aticas,
Pontificia \allowbreak Universidad \allowbreak{Cat\'olica} de Chile, Santiago, Chile.}
\email{gianurzua@gmail.com}

\begin{document}

\maketitle

\begin{abstract}
Wahl's conjecture states that two-dimensional singularities admitting a rational homology disk smoothing are weighted homogeneous. Assuming the conjecture, we classify all normal degenerations of the complex projective plane with only rational singularities. They are precisely the  surfaces classified by Manetti and Hacking--Prokhorov, which are controlled by the Markov equation, together with six new degenerations containing four non-log canonical singularities.
\end{abstract}

\tableofcontents

\section{Introduction} \label{s0}

In \cite{Bad86}, B\u adescu studied normal degenerations of rational surfaces with $K^2 \geq 0$. More precisely, he considered proper deformations $$(W \subset \mathcal{W}) \to (0 \in \D)$$ over a disk $\D$, where $W_t$ is a nonsingular rational surface with $K_{W_t}^2 \geq 0$ for $t \neq 0$, and $W:=W_0$ is a normal surface. He proved that either (A) $W$ is a rational surface with at most rational singularities, or (B) $W$ has precisely one non-rational singularity and satisfies some particular properties (see \cite[Thm.~1]{Bad86} for details). When $W_t=\P^2$, he showed that if $K_W$ is Cartier, then $W$ is either $\P^2$ or the elliptic cone over a nonsingular plane cubic \cite[(4.3)~Thm.]{Bad86} (which is the only strictly log canonical degeneration of $\P_{\C}^2$ by \cite[Thm.~8.5]{Hack04}). B\u adescu also pointed out the classical construction of ``sweeping out the cone'' over nonsingular plane curves as a source of further examples, and asked whether these are all the normal degenerations of $\P^2$ \cite[(4.8.2)]{Bad86}. This was not the case. In \cite{Man91}, Manetti found infinitely many degenerations of $\P^2$ of type (A). He proved in \cite[Thm.~11]{Man91} that for type (A) such degenerations $W$ must have either (A1) only quotient singularities or (A2) at most one non-quotient singularity. For (A1), he proved that the singularities of $W$ are Wahl singularities, at most three, and that the degeneration is $\Q$-Gorenstein \cite[Main Thm.]{Man91}(see also \cite[III and IV]{ManTh}). Finally, Hacking and Prokhorov in \cite{HP10} proved that any degeneration of type (A1) is a $\Q$-Gorenstein partial smoothing of a $\P(a^2,b^2,c^2)$, where $(a,b,c)$ satisfies the Markov equation $$a^2+b^2+c^2=3abc.$$ 

A \QHD singularity is a two-dimensional normal singularity that admits a rational homology disk smoothing, that is, a smoothing with Milnor number equal to $0$. A well-known conjecture of J. Wahl \cite{Wahl_2011,Wahl_21} states that every \QHD singularity is weighted homogeneous \cite{PSS14,Be25}. By \cite{SSW_2008} and \cite{BS_2011}, there is an explicit classification of weighted homogeneous \QHD singularities. We refer to their classification by the star dual graphs of their minimal resolutions, as shown in Figures \ref{fig QHD V3} and \ref{fig QHD V4}. Under the assumption that Wahl's conjecture holds, in this paper we classify all degenerations of type (A2), and so we finish the classification of normal degenerations of $\P_{\C}^2$ with only rational singularities.

\begin{theorem}
Assume that every \QHD singularity is weighted homogeneous. Then every normal degeneration of $\P_{\C}^2$ with only rational singularities is either a partial $\Q$-Gorenstein smoothing of $\P(a^2,b^2,c^2)$, where $a^2+b^2+c^2=3abc$, or one of the six surfaces $W_j$, $W_f$, $W_c$, $W_{c5}$, $W_b$, and $W_{b13}$ shown in Figure \ref{mainfig}.
\label{main}
\end{theorem}

\begin{figure}[ht]
    \centering
    \includegraphics[scale=0.97]{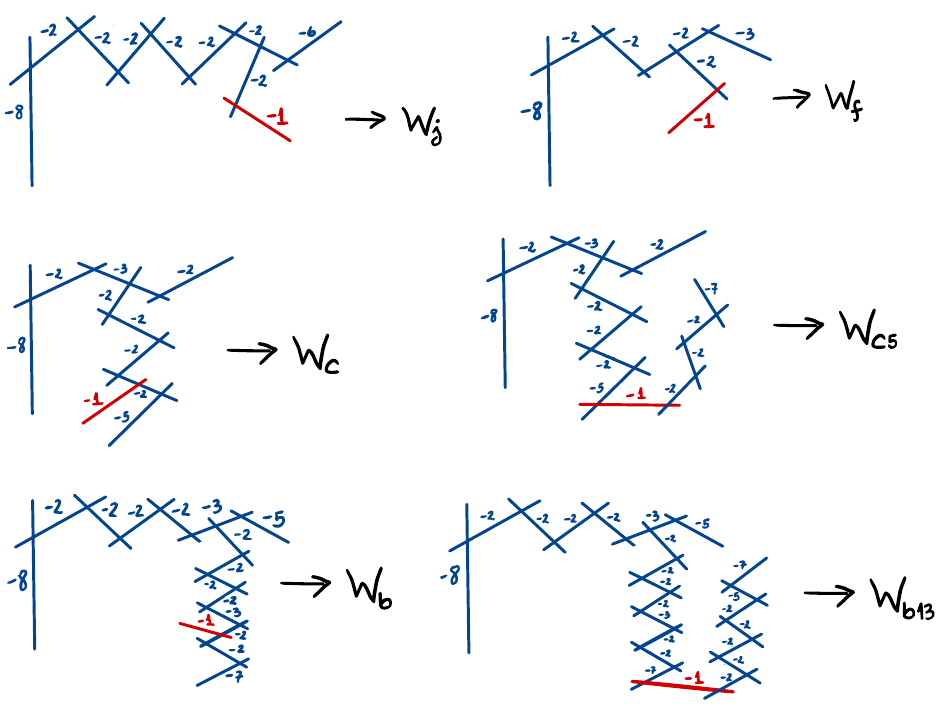}
    \caption{The minimal resolutions of the six non-log canonical rational degenerations of $\P_{\C}^2$. All of them blow down to the Hirzebruch surface $\F_8$.}
    \label{mainfig}
\end{figure}

Thus, there are four new singularities arising in degenerations of $\P_{\C}^2$. They are the \QHD singularities of type (b) with $p=2$, $q=r=4$ ($\mathcal{N}(3,4,4)$); type (c) with $r=4$, $q=1$ ($\mathcal{N}(0,1,4)$); type (f) with $q=2$ ($\mathcal{M}(0,2,0)$); and type (j) with $q=4$ ($\mathcal{C}_3^3(4)$) \footnote{In parentheses is the original notation from \cite{SSW_2008,Wahl_21}.}. Their indices are $58$, $9$, $8$, and $16$, respectively. In particular, there are normal degenerations of $\P_{\C}^2$ whose indices are not Markov numbers (but they remain consistent with the conjectures in \cite{DeVS24}). The classification also shows that degenerations of type (A2) have at most $3$ singularities, since the new ones have only $1$ or $2$ singularities (see \cite[Cor.~11]{Man93}). In fact, four of them have a single singularity, while the remaining two are obtained by applying slidings \cite[Def.~3.17]{CU26} to two of the former, producing two common degenerations with the Manetti and Hacking--Prokhorov family at the singularities $\frac{1}{5^2}(1,4)$ and $\frac{1}{13^2}(1,25)$.

Consider one of the new degenerations $(W \subset \mathcal{W}) \to (0 \in \D)$. Then its Seifert blow-up \cite{Wahl_2013} defines an slc degeneration of $\P_{\C}^2$ formed by two surfaces, $W'$ (the proper transform of $W$) and $Z$ (the compactified Milnor fiber), glued along a $\P_{\C}^1$ with three orbifold double normal crossing singularities. Its canonical class is not ample (compare with \cite[Lem.~9.5]{Hack04}), and it defines an extremal flipping neighborhood $(\Gamma^- \subset \W) \to (q \in \Y)$.

\begin{theorem}
Let $(W^+ \cup Z^+ \subset \W^+) \to (0 \in \D)$ be the degeneration obtained by flipping $(\Gamma^- \subset \W) \to (q \in \Y)$. Then there is a divisorial contraction
$(W^+ \cup Z^+ \subset \W^+) \to (\bar{Z} \subset \bar{\W})$
contracting $W^+$ such that the induced degeneration $(\bar{Z} \subset \bar{\W}) \to (0 \in \D)$ is $\Q$-Gorenstein and satisfies:
\begin{itemize}
    \item[(1)] if $W$ has one singularity, then $\bar{Z}=\P_{\C}^2$;
    
    \item[(2)] if $W=W_{c5}$, then $\bar{Z}$ has only the singularity $\frac{1}{5^2}(1,4)$;
    
    \item[(3)] if $W=W_{b13}$, then $\bar{Z}$ has only the singularity $\frac{1}{13^2}(1,25)$.
\end{itemize}
    \label{flip}
\end{theorem}

We also construct a normal projective surface $W$ with $-K_W$ ample, $K_W^2=9$, $b_2(W)=1$, and one of the rational singularities studied by Beke, Plamenevskaya, and Starkston in \cite{BPS26}. This singularity admits a rational homology disk Stein filling \cite[Thm.~1.1]{BPS26} but no \QHD smoothing. It symplectically rationally blows down to $\P_{\C}^2$ (Theorem \ref{SRBDofBPS}).

\subsubsection*{Acknowledgments} The first-named author was funded by the ANID doctoral scholarship 21220497, and the second-named author by FONDECYT regular grant 1230065. \cite{progrCanedo}

\section{Preliminaries} \label{s1}

We will be using notation and some facts from Sections 2 and 3 of \cite{CU26}. Let $(W \subset \mathcal{W}) \to (0 \in \D)$ be a normal degeneration of $\P_{\C}^2$ of type (A2) with a rational non-quotient singularity. Let $l$ be the number of singularities of $W$. Consider the diagram
$$\xymatrix{  & X  \ar[ld]_{\pi} \ar[rd]^{\phi} &  \\ S &  & W}$$
where the morphism $\phi$ is the minimal resolution of $W$, and $\pi$ is a composition of $m$ blow-ups such that $S$ has no $(-1)$-curves. Let $E_i$ be the pull-back divisor in $X$ of the $i$-th point blown-up through $\pi$. Thus, $E_i$ is a connected, possibly non-reduced tree of $\P_{\C}^1$s, $E_i^2=-1$, and $E_i\cdot E_j=0$ for $i\neq j$. Let $E=\sum_{i=1}^{m}{E_i}$. We have $K_X \sim \pi^{*}(K_S)+E$. Let $C=\sum_{i=1}^l C_i$ be the exceptional (reduced) divisor of $\phi$, where the $C_i$ are the exceptional divisors over each singularity. Let $|C_i|$ be the number of curves in $C_i$. 

By \cite[Thm.~1]{Bad86}, we know that $W$ is rational. As in \cite{Man91,ManTh}, let us choose $\pi \colon X \to S=\F_d$ with maximal $d \geq 2$. We recall that the $\P_{\C}^1$-bundle $\F_d \to \P_{\C}^1$ has a unique negative curve $\sigma_{\infty}$ with $\sigma_{\infty}^2=-d$. Let $F$ be the class of a fiber. We have an induced fibration $\pi' \colon X \to \P_{\C}^1$ via composition of $\pi$ with $\F_d \to \P_{\C}^1$. Due to our maximal choice, there are no blow-ups over $\sigma_{\infty}$.   

\begin{remark}
We have the following known facts for $W$.

\begin{itemize}
\item[(R1)] By \cite[Prop.~3]{Man91}, we have that  $(W \subset \mathcal{W}) \to (0 \in \D)$ is projective, $q(W)=h^1(\O_W)=0$ and $p_g(W)=h^2(\O_W)=0$.

\item[(R2)] By \cite[Thm.~11]{Man91}, we have at most one non-quotient singularity in $W$, and $C$ consists of $\sigma_{\infty} \subset X$ plus some irreducible components of fibers of $\pi'$. Let $C_1$ be the component over the non-quotient singularity. Then $\sigma_{\infty} \subset C_1$.  

\item[(R3)] By \cite[Cor.~5]{Man91}, the canonical class $K_{\mathcal{W}}$ is $\Q$-Cartier, and $K_W^2=9$. In particular, any quotient singularity of $W$ must be Wahl and $(W \subset \mathcal{W}) \to (0 \in \D)$ is a $\Q$-Gorenstein smoothing.

\item[(R4)] By \cite[Thm.~1]{Bad86}, we have that $\rho(W)=1$. Hence, if $F_i$ is a fiber of $\pi'$ with components in $C$, then there is exactly one $(-1)$-curve in $F_i$ and any other curve in $F_i$ belongs to $C$. Let $\# f$ be the number of such fibers. We have that $-K_W$ is ample.

\item[(R5)] The tangent sheaf of $W$ is $T_W:=\mathcal{H}\text{om}_{\O_W}(\Omega_W^1,\O_W)$. By \cite[Lem.~11]{Man93}, we have $H^2(W,T_W)=0$, and so there are no local-to-global obstructions to deform $W$. 

\end{itemize}
\label{initialfacts}
\end{remark}

\begin{definition}
For any given $E_i$, we define the sets of divisors 
$$C_{out,i}=\{ \Gamma \in C \colon \Gamma \nsubseteq E_i \} \ \ \ \text{and} \ \ \ C_{in,i}=\{ \Gamma \in C \colon \Gamma \subseteq E_i \}.$$ Let $S_h$ be the set of $E_i$s such that $E_i \cdot C_{out,i}=h$. Let $T_h$ be the set of $E_i$s in $S_h$ that satisfies $E_i \cdot C_{in,i}=0$. Let $s_h=|S_h|$ and $t_h=|T_h|$ be the corresponding cardinalities.
\label{def ShTh}
\end{definition}

For any divisor $D$ on a nonsingular projective surface $Y$, we define $$q(D):=-\frac{1}{2}(D^2+ D\cdot K_Y).$$ If $D$ is effective, then $q(D)=\chi(\O_D)$. 

\begin{lemma}
We have $s_0=0$, $s_i=0$ for $i\geq 3$, $\# f = t_1+t_2$, and $q(E+C)=1+s_1-t_1$.
\label{lemmaq}
\end{lemma}

\begin{proof}
For $E_k$ there are two possibilities: either it is the last $(-1)$-curve over a fiber or it is not. In the first case, $E_k \in T_1 \cup T_2$. In the second case, $E_k \in S_h$ for some $h=1,2$, by the projection formula and the fact that $\pi^*(\pi(C))_{\text{red}}$ is a simple normal crossings divisor. By \cite[Rem.~3.4]{CU26}, we have $q(E+C)=q(\pi(C))+ s_1-t_1$, and directly $q(\pi(C))=1$.
\end{proof}

\begin{lemma}
For each singularity of $W$, the Milnor number of the smoothing induced by $(W \subset \mathcal{W}) \to (0 \in \D)$ is equal to $0$.
\label{mu=0}
\end{lemma}

\begin{proof}
Let $p_1, \ldots, p_l$ be the singular points of $W$, and let $\mu_i$ be the Milnor number of the smoothing induced by $(W \subset \mathcal{W}) \to (0 \in \D)$ at $p_i$. By definition, $\mu_i$ is the second Betti number of the corresponding Milnor fiber. The minimal resolution $X \to W$ and the smoothing give the following relations between topological Euler characteristics: $$\chi_{top}(X)=\chi_{top}(W \setminus \cup_{i=1}^l p_i) + \sum_{i=1}^l \chi_{top}(C_i)= \chi_{top}(W_t) -  \sum_{i=1}^l (1+\mu_i) + \sum_{i=1}^l \chi_{top}(C_i).$$ Here we are using that the first Betti number of the Milnor fibers of a smoothing is always $0$, this is a result of Greuel--Steenbrink \cite{GS_1983}. On the other hand, by replacing $\chi(\O_W)=\chi(\O_{W_t})$ and the Noether formula, we obtain  $$K_{W_t}^2 + \sum_{i=1}^l \mu_i = 10 \chi(\O_X)-\chi_{top}(X) + \sum_{i=1}^l \big(b_2(C_i)-b_1(C_i)\big) + 12 \sum_{i=1}^l p_g(p_i \in W),$$ where $p_g(p_i \in W) := \dim_{\C} H^0(\bar X,(R^1 \pi_* \O_X)_{p_i})$ is the geometric genus of the singularity. As all the singularities of $W$ are rational, $W$ is rational, and $\rho(W)=1$, we obtain  $K_{W_t}^2 + \sum_{i=1}^l \mu_i = 10-(\rho(X) - \sum_{i=1}^l |C_i|)=10-1=9$. But $K_{W_t}^2=9$, and so $\mu_i=0$ for all $i$.
\end{proof}

\begin{figure}[htbp]
    \centering
    \includegraphics[scale=0.51]{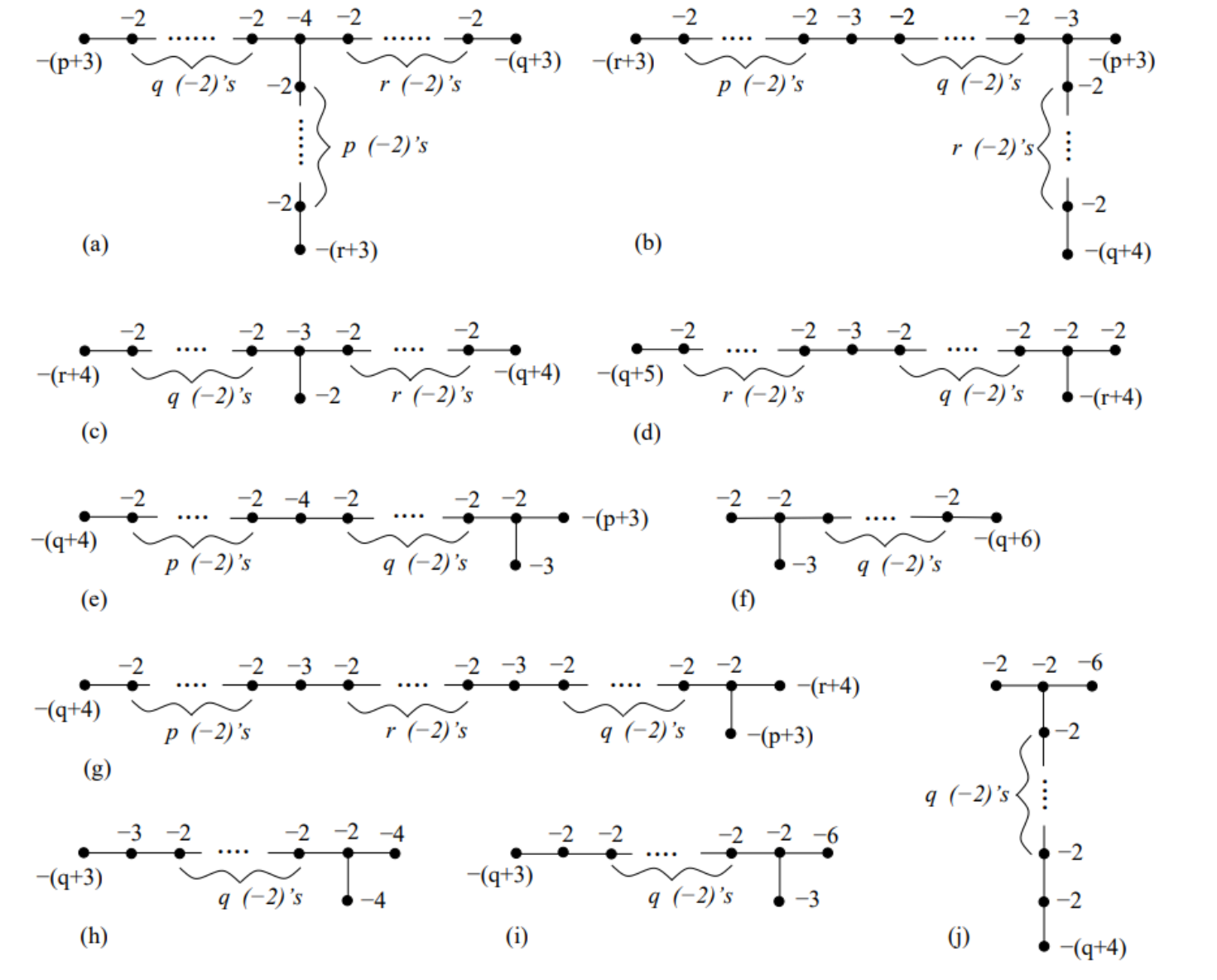}  
    \caption{Star graphs of \QHD singularities of valency 3 from \cite{BS_2011}.}
    \label{fig QHD V3}
\end{figure}

\begin{figure}[htbp]
    \centering
    \includegraphics[scale=0.31]{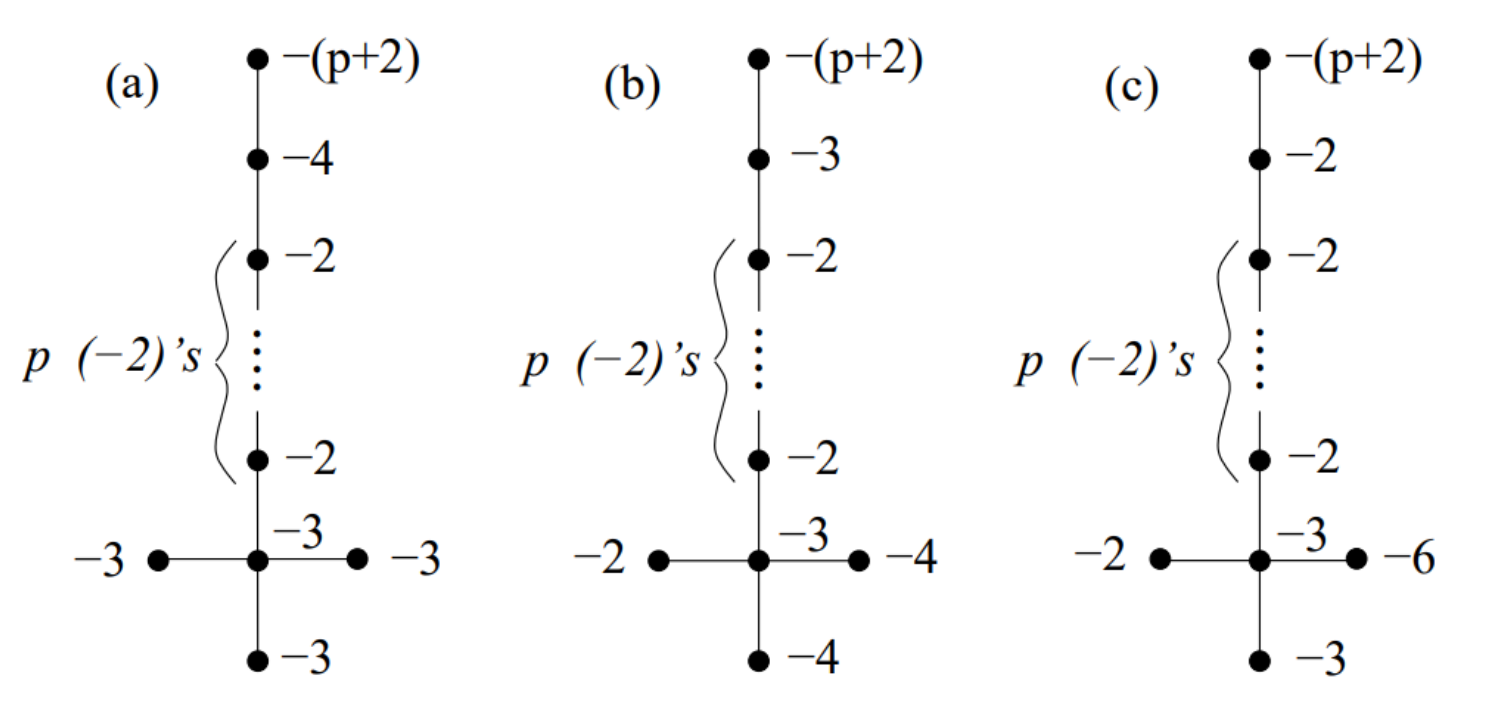}
    \caption{Star graphs of \QHD singularities of valency 4 from \cite{BS_2011}.}
    \label{fig QHD V4}
\end{figure}

In this way, the singularities of $W$ are \QHD. From now on, we assume Wahl's conjecture, i.e. that \QHD singularities are weighted homogeneous, and refer to their classification by the star dual graphs of their minimal resolutions shown in Figures \ref{fig QHD V3} and \ref{fig QHD V4} (see \cite{progrCanedo} for some computations). We call the curve in the star graph meeting three or four other curves the \textit{central curve}, and the chains of curves emanating from it the \textit{legs}.

\begin{lemma}
We have $d=4+2 \# f+s_1-t_1$ and no valency $4$ \QHD singularity in $W$.
\label{dformula}
\end{lemma}

\begin{proof}
By \cite[Prop.~3.5]{CU26}, we have $K_W^2-K_S^2=K_S \cdot \pi(C)-q(E+C)+v_4$, where $v_4$ is the number of valency $4$ \QHD singularities in $W$. We have $K_W^2-K_S^2=1$, $K_S \cdot \pi(C)=d-2-2 \# f$, and $q(E+C)=1+s_1-t_1$ by Lemma \ref{lemmaq}. Therefore $d=4+2 \# f +s_1-t_1 -v_4 \geq 6$ as $\# f \geq 1$, $s_1 - t_1 \geq 1$ and $v_4=0$ or $1$ by Remark \ref{initialfacts} (R2). By classification of \QHD and Remark \ref{initialfacts} (R2) again, we have that $\sigma_{\infty}$ is at the end of some leg of $C_1$ (star graph of the non-quotient \QHD singularity). The central curve lies in a fiber of $\pi' \colon X \to \P_{\C}^1$, so $v_4=0$; otherwise that fiber would contain at least two $(-1)$-curves.
\end{proof}

\begin{corollary}
We have $\# f=1$, one leg of $C_1$ has to be completely contracted by $\pi$ before the others, $s_1-t_1 \geq 2$ and $d=6+s_1-t_1 \geq 8$.
\label{dcorollary}
\end{corollary}

\begin{proof}
As $\sigma_{\infty}$ is at the end of a leg, we must have $\# f=1$ as we cannot blow-up over $\sigma_{\infty}$ and there is at most one $(-1)$-curve in fibers of $\pi'$. As we have only one $(-1)$-curve in a fiber of $\pi'$, one leg of $C_1$ has to be contracted completely before the others, and this creates another element in $S_1 \setminus T_1$, and so $s_1-t_1 \geq 2$. The last statement is direct from Lemma \ref{dformula}.   
\end{proof}

\section{Classification} \label{s2}

We now determine all possible minimal resolutions $X \to W$ satisfying the constraints established in the previous section. By Corollary \ref{dcorollary}, we start with a Hirzebruch surface $\F_d$ with $d\geq 8$, and blow up points over one fiber of $\F_d \to \P_{\C}^1$ until we obtain the longest possible chain over that fiber. This is represented in Figure \ref{two legs}, where $X' \to \F_d$ denotes the corresponding composition of blow-ups. Let $G \subset X'$ be the $(-1)$-curve over the chosen fiber. In terms of Hirzebruch--Jung continued fractions, we have $$[a_s,\ldots,a_1,1,b_1,\ldots,b_t]=0,$$ where $a_i,b_j \geq 2$. Thus we can read the $a_i$s from the $b_j$s.

\begin{figure}[htbp]
    \centering
    \includegraphics[scale=1.1]{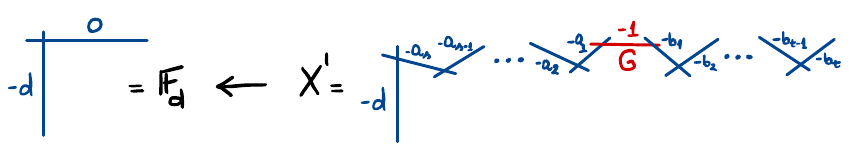}
    \caption{\small{Two legs of $C_1$.}}
    \label{two legs}
\end{figure}

\begin{lemma}
The proper transform of $G$ in $X$ is the central curve of $C_1$.
\label{centralcurve}
\end{lemma}

\begin{proof}
The next blow-up over $X'$ must be at a point of $G$ not meeting either of the chains $a_i$ or $b_i$, since the chain is maximal. We then construct a new maximal chain over $G$. If the proper transform of $G$ intersects the new $(-1)$-curve $G'$, then the next blow-up must be over $G'$, since otherwise all exceptional divisors of $\phi$ would form chains. In this case, $G$ is the central curve. On the other hand, if the proper transform of $G$ does not intersect the new $(-1)$-curve, then it must be the central curve. \end{proof}

The third leg of $C_1$ emanating from $G$ will be called the \textit{contractible leg} (the one appearing in Corollary \ref{dcorollary}).

\begin{lemma}
If the contractible leg consists of one curve, then it is a $(-2)$-curve. 
\label{lemma1}
\end{lemma}

\begin{proof}
Since $G$ is the central curve of $C_1$, its self-intersection must be $-2$, $-3$, or $-4$ by the classification of \QHD singularities. Over $G$, we have either the contractible leg together with a $(-1)$-curve, or the contractible leg and a Wahl chain joined by a $(-1)$-curve. If $G^2=-4$, then the residual Wahl chain is $[2,\ldots,2,3,2]$, which is never a Wahl chain. If $G^2=-3$, then the residual Wahl chain is $[2,\ldots,2,3]$, which again is never Wahl. Finally, if $G^2=-2$, then the only possibility is that the contractible leg is a single $(-2)$-curve and there is no Wahl chain. 
\end{proof}

\begin{lemma}
The only possible legs $[d,a_s,\ldots,a_1]$ and $[b_1,\ldots,b_t]$ of $C_1$ and \QHD types (from the classification in Figure \ref{fig QHD V3}) are:
\begin{itemize}
\item[(L1)] $[d,\underbrace{2,\ldots,2}_{\alpha}]$ and $[\alpha+1]$, with $\alpha \geq 1$; types (a), (b), (c), (f), (j).
\item[(L2)] $[d,\underbrace{2,\ldots,2}_{\beta},3,\underbrace{2,\ldots,2}_{\alpha}]$ and $[\alpha+2,\beta+2]$, with $\alpha, \beta \geq 0$; not possible.
\item[(L3)] $[d,\underbrace{2,\ldots,2}_{\beta},4,\underbrace{2,\ldots,2}_{\alpha}]$ and $[\alpha+2,2,\beta+2]$, with $\alpha, \beta \geq 0$; not possible.
\item[(L4)] $[d,\underbrace{2,\ldots,2}_{\gamma},3,\underbrace{2,\ldots,2}_{\beta},3,\underbrace{2,\ldots,2}_{\alpha}]$ and $[\alpha+2,\beta+3,\gamma+2]$, with $\alpha, \beta, \gamma \geq 0$; not possible.
\end{itemize}
\label{possiblelegs} 
\end{lemma}

\begin{proof}
First, we must have $s,t \geq 1$, because there is only one $(-1)$-curve over the fiber. By classification of \QHD star graphs of valency $3$, we have listed all possibilities for the leg $[d,a_s,\ldots,a_1]$. After that, we have listed the corresponding duals $[b_1,\ldots,b_t]$. For (L4), the only possible is $(g)$. By Lemma \ref{lemma1}, this is not possible. Same for (L3). For (L2) we have types (h), (d), (b). For (h) same argument. For (d), we have contradiction because one leg has $2$ curves. For (b), we have $r=1$, but then $d=r+3=4$ a contradiction. For (L1) we have types (a), (b), (c), (f), (i), (j), but (i) is impossible using Lemma \ref{lemma1}.
\end{proof}

\begin{proof}[Proof of Theorem \ref{main}]
Given the degeneration $(W \subset \mathcal{W}) \to (0 \in \D)$ with only rational singularities, Lemma \ref{mu=0} implies that $W$ has only \QHD singularities. Assuming Wahl's conjecture, these singularities are weighted homogeneous. If they are log canonical, then they are log terminal by \cite{Hack04}, and the degenerations are classified by Hacking--Prokhorov \cite[Cor.~1.2]{HP10}. Otherwise, we are left with the possibilities listed in Lemma \ref{possiblelegs} (L1), which we now analyze with the same notation and the $p,q,r$ notation in Figure \ref{fig QHD V3}.

\textbf{Type (a):} We have $d=p+3$, and $\alpha+1=q+3$ or $r+3$. If $\alpha+1=r+3$, then $p=0$, whereas $d\geq 8$, a contradiction. Hence $\alpha=q+2$. On the other hand, we also have $\alpha=q$, which is impossible. Thus, this type cannot occur.

\textbf{Type (b):} We have $d=q+4 \geq 8$ and $\alpha+1=p+3$, so $\alpha=p+2$ and hence $r=\alpha=p+2$. There are two cases for the contractible leg: (I) $[\underbrace{2,\ldots,2}{q},3,\underbrace{2,\ldots,2}{p},p+5]$, or (II) the same leg followed by a $(-1)$-curve and a Wahl chain. In case (I), we must subtract $1$ from some entry in the leg and continue blowing down until we reach $G$. This subtraction must occur on a $2$ in the leg. It cannot be one of the first $q$ entries, since $q \geq 4$, nor can it be one of the two rightmost $2$s in $[\ldots,\underbrace{2,\ldots,2}_{p},\ldots]$. Hence the only possibility is the leftmost $2$ in this block. Then the only possibility is $p=2$, which forces $q=4$. This is type (b) with $p=2$ and $q=r=4$. For (II), one checks by the duality of continued fractions that the Wahl chain must have the form $[q+3,p+3,\underbrace{2,\ldots,2}_{p+3}]$. Therefore, $p=2$ and $q=4$. 

\textbf{Type (c):} We have $d=q+4 \geq 8$ and $\alpha+1=2$ or $r+4$. In the latter case, we must have $q=0$, which is impossible. Thus $\alpha=1=r$. As before, there are two cases for the contractible leg: (I) $[\underbrace{2,\ldots,2}_{q},5]$, or (II) the same leg followed by a $(-1)$-curve and a Wahl chain. In case (I), we can only subtract $1$ from the rightmost $2$, forcing $q=4$. This is type (c) $q=4$, $r=1$. In case (II), again by the duality of continued fractions, the Wahl chain must have the form $[2,2,2,q+3]$, which again forces $q=4$.

\textbf{Type (f):} We have $d=q+6$ and, by Lemma \ref{lemma1}, $\alpha+1=3$. Hence $\alpha=2=q$, and this is type (f) with $q=2$.

\textbf{Type (j):} We have $d=q+4$ and, by Lemma \ref{lemma1}, $\alpha+1=6$, and so this is type (j) with $q=4$.

These are all the combinatorial possibilities. On the other hand, by Artin's contractibility criterion, starting with $\F_8$ we can contract the corresponding configurations in the composition of blow-ups $X \to \F_8$, obtaining the six surfaces shown in Figure \ref{mainfig}. Let $\phi \colon X \to W$ denote the contraction. We now prove the existence of a smoothing $(W \subset \mathcal{W}) \to (0 \in \D)$ with $W_t=\P_{\C}^2$ for $t \neq 0$. By \cite[Thm.~5.1]{CU26}, there are no local-to-global obstructions to deforming $W$, since the unique non-Wahl singularity is a valency $3$ \QHD singularity (see \cite[Rem.~5.2]{CU26}). As each singularity of $W$ admits a \QHD smoothing, these local smoothings globalize to give the desired smoothing of $W$. Wahl \cite{Wahl_2013} proved that $\mathcal{W}$ is log terminal and $\Q$-Gorenstein, that is, it is locally the quotient of an equivariant deformation of the canonical cover of the singularities of $W$. In particular, $-mK_{\mathcal{W}}$ is Cartier for some $m \in \Z_{>0}$. Since $K_W^2=9$ and $-K_W$ is ample, it follows that $K_{W_t}^2=9$ and $-K_{W_t}$ is ample by \cite[Prop.~1.41]{KM_1998}. Therefore $W_t=\P_{\C}^2$.
\end{proof}

\begin{remark}
It is straightforward to verify that the image of the $(-1)$-curve in the minimal resolutions of $W_c$ and $W_b$ (see Figure \ref{mainfig}) defines slidings \cite[Def.~3.17]{CU26} to the surfaces $W_{c5}$ and $W_{b13}$, respectively. Hence, at least for these two surfaces, there is a common degeneration between the log terminal and non-log terminal rational degenerations of $\P_{\C}^2$.
\label{theslidings}
\end{remark}

\section{MMP over the new degenerations} \label{s3}

Let $W$ be one of the six surfaces in Theorem \ref{main} with a non-log terminal singularity $(p \in W)$. Let $(W \subset \mathcal{W}) \to (0 \in \D)$ be a \QHD smoothing. As mentioned above, Wahl \cite{Wahl_2013} proved that $\mathcal{W}$ is log terminal and $\Q$-Gorenstein. He also showed the existence of a weighted blow-up at $(p \in \mathcal{W})$, which is called the Seifert partial resolution of $\mathcal{W}$ at $p$ (see \cite[Def.~5.3]{CU26}), producing a new degeneration over $\D$
$$(W' \cup Z) \subset \mathcal{W}' \to (W \subset \mathcal{W}) \to (0 \in \D),$$
such that (see \cite[Sect.~7]{CU26}):
\begin{itemize}
    \item $W'$ and $Z$ are projective surfaces intersecting along a curve $\Delta\simeq \P_{\C}^1$.
    \item The restriction $(\Delta \subset W') \to (p \in W)$ is the Seifert partial resolution of the \QHD singularity $(p \in W)$, extracting only the central curve $\Delta$ of the star resolution. The surface $W'$ has three cyclic quotient singularities $\frac{1}{m_i}(1,q_i)$ corresponding to the three legs of the star dual graph.
    \item The surface $Z$ is the compactified Milnor fiber of the \QHD smoothing of $(p \in W)$ (see \cite[Def.~5.5]{CU26}). It has three cyclic quotient singularities $\frac{1}{m_i}(1,m_i-q_i)$ lying on $\Delta$ at the same points as those on $W'$. The fiber $W' \cup Z$ has orbifold double normal crossing singularities as in \cite[Sect.~5]{Hack12}.
\end{itemize}

Let $\Gamma^-$ be the image in $W$ of the $(-1)$-curve in the fiber of $\pi' \colon X \to \P_{\C}^1$. Let us denote by the same letter its proper transform in $W'$.

\begin{lemma}
The curve $\Gamma^- \subset W'$ induces an extremal neighborhood of flipping type $(\Gamma^- \subset \W) \to (q \in \Y)$ between $3$-folds over $\D$. It is semistable of type k1A (if $W$ has one singularity) or k2A (if $W$ has two singularities), as in \cite{HTU17}.
\label{nbhd}
\end{lemma}

\begin{proof}
First, for each $W$ we have that $\Gamma^- \cdot K_{\W'}=\Gamma^- \cdot (K_{W'} + {\Delta}|_{W'})$ is equal to $$-\frac{1}{2}, -\frac{1}{2}, -\frac{4}{21}, -\frac{11}{144}, -\frac{4}{21 \cdot 5},  -\frac{11}{144 \cdot 13}$$ for the types (j), (f), (c), (b), (c5) and (b13), respectively. Moreover, we have ${\Gamma^-}^2<0$ as $\Gamma^-$ is contractible in $W'$, and so it is an extremal ray for $K_{W'}+{\Delta}|_{W'}$. We now follow steps (1)--(7) in \cite[Sect.~7]{CU26}. The only subtlety arises for semistable $k1A$ extremal neighborhoods. Since there is a birational morphism $\W' \to \W$ over $\D$ contracting only $Z$, the Milnor fiber around $\Gamma^-$ is the same as the Milnor fiber of the smoothing around $\Gamma^- \subset W$. Its second Betti number is equal to $1$, since $\Gamma^- \simeq \P_{\C}^1$ and the singularity admits a \QHD smoothing. Therefore, we obtain a semistable flipping extremal neighborhood of type $k1A$ with second Betti number equal to $1$. By \cite[Prop.~2.1]{HTU17}, it is one of the $k1A$ neighborhoods in \cite{HTU17}.
\end{proof}

\begin{proof}[Proof of Theorem \ref{flip}]
First, we note that the flip $(\Gamma^+ \subset \W^+) \to (q \in \Y)$ has the same local description around $\Gamma^+$ in both the $k1A$ and $k2A$ cases, since $k1A$ and $k2A$ extremal neighborhoods are deformation equivalent by \cite{HTU17}. Therefore, we may regard it as the flip of a $k2A$ extremal neighborhood. We then apply directly the argument of step (7) in \cite[Sect.~7]{CU26}. The flipped curve $\Gamma^+$ is contained in $Z^+$ and not in $W^+$. Note that, for each of the six possible surfaces $W$, contracting $\Gamma^-$ produces a nonsingular point on the image surface in $\Y$, and hence on $W^+$. Consequently, the corresponding point on $Z^+$ is also nonsingular. Straightforward computations give the list of the self-intersection of the flipped curves $\Gamma^+$ and the additional Wahl chain, corresponding to the Wahl singularity appearing after the flip for each of the six surfaces $W$. We also record the continued fraction of the two cyclic quotient singularities on $Z^+$ lying on $\Delta$, where $\Delta^2=-1$.

\bigskip 
\noindent
(j): ${\Gamma^+}^2=-2$ and no Wahl chain; $[2,2,2,2,2,2,7,1,2,2,2,2,2]$.

\noindent
(f): ${\Gamma^+}=-2$ and no Wahl chain; $[2,2,2,2,2,2,4,1,2,2]$.

\noindent
(c) and (c5): ${\Gamma^+}^2=-1$ and $[2,2,2,7]$; $[2,2,2,2,2,2,3,1,2]$.

\noindent
(b) and (b13): ${\Gamma^+}^2=-1$ and $[7,5,2,2,2,2,2]$; $[2,2,2,2,2,2,6,1,2,2,2,2]$.
\bigskip 

Since $\rho(W^+)=1$, we have that $K_{W^+}+\Delta^+$ is ample because $\Delta^+ \cdot (K_{W^+}+\Delta^+)=-\frac{1}{a}-\frac{1}{b}<0$, where $a$ and $b$ are the orders of the two cyclic quotient singularities on $W^+$. We have $(\Delta^+)^2>0$ in $W^+$, whereas $(\Delta^+)^2<0$ in $Z^+$. Hence
$\Delta^+ \cdot K_{\W^+}=\Delta^+ \cdot (K_{W^+}+\Delta^+)<0$,
and therefore (as in the proof of \cite[Lem.~9.5]{Hack04}) there is a divisorial contraction
$(W^+ \cup Z^+ \subset \W^+) \to (\bar{Z} \subset \bar{\W})$ contracting $W^+$. In all six cases, the contraction of $\Delta^+$ in $Z^+$ is a smooth point, as follows from the continued fractions listed above. Moreover, $\rho(Z^+)=2$, and hence $\rho(\bar{Z})=1$, giving the claimed list. The resulting family must be $\Q$-Gorenstein.
\end{proof}

\section{BPS example} \label{s4}

Recently, Beke, Plamenevskaya and Starkston \cite{BPS26} constructed an infinite family of rational singularities which are not weighted homegeneous, but their minimal resolutions admit a symplectic rational blow-down. These examples admit a Stein filling which is a rational homology disk, but do not admit \QHD smoothings. 

Suppose that some of these examples satisfy Remark \ref{initialfacts}. Then it is easy to verify that the configuration in Figure \ref{fig bps} is the only possibility (in \cite{BPS26}'s notation, $k=1$ and $n=7$). Indeed, the $(-1)$-curve must intersect a $(-2)$-curve transversally at a single point, and there is only one possible position. Thus, we again start with $\F_8$ and blow up over one fiber to obtain $X$. By Artin's contractibility criterion, one obtains a singular surface $W_{BPS}$ with a single rational singularity, $-K_{W_{BPS}}$ ample, $\rho(W_{BPS})=1$, and $K_{W_{BPS}}^2=9$. The surface $W_{BPS}$ is not smoothable to $\P_{\C}^2$.

\begin{figure}[ht]
\centering
\includegraphics[scale=0.9]{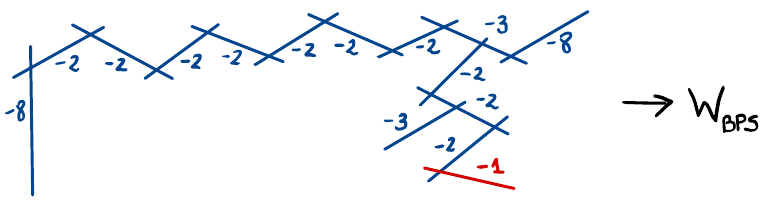}
\caption{The non-smoothable BPS example.}
\label{fig bps}
\end{figure}

\begin{theorem}
The symplectic rational blow-down of $X$ along the BPS configuration in $X$ (shown in Figure \ref{fig bps}) is $\P_{\C}^2$.
\label{SRBDofBPS}
\end{theorem}


\begin{proof}

Just as in the proof of \cite[Thm.~2.3]{RU25}, we can trace the symplectic form $\omega$ on the symplectic rational blow-down $Y$ of the configuration in $X$ (the minimal resolution of $W:=W_{BPS}$). It follows that $[\omega] = -M K_W$ for some $M>0$. We recall that $-K_W$ is ample. On the other hand, since the rational blow-down replaces the configuration by a rational homology disk, we have $H_2(Y) \subset H_2(W)$ with finite index. Moreover, the symplectic canonical class $K_Y$ is represented by $K_W$, and hence $K_Y \cdot [\omega] < 0$. By a result of Liu and Ohta--Ono (see \cite[Cor.~1.4(ii)]{McS96}), it follows that $Y$ is rational. Since $b_2(Y)=1$, we conclude that $Y=\P_{\C}^2$.
\end{proof}

 

\bibliographystyle{abbrv}
\bibliography{Bibliography}
\end{document}